\documentclass[12pt]{article} 
\usepackage{amsmath}
\usepackage{amssymb}
\usepackage{amsthm}
\usepackage{mathtools}

\usepackage[utf8]{inputenc}
\usepackage[english]{babel}
\usepackage{color}

\usepackage{wrapfig}

\newtheorem{theorem}{Theorem}[section]
\newtheorem{corollary}[theorem]{Corollary}

\newtheorem{lemma}[theorem]{Lemma}

\newtheorem{proposition}[theorem]{Proposition}

\title{An incomplete attack on the upper bound of the unit distance problem}
\date{\today}
\author{Steven Senger}

\begin{document}
\maketitle
\begin{abstract}
This is an incomplete attempt to show that the upper bound of $\lesssim n^\frac{4}{3}$ on the number unit distances determined by a large finite set of $n$ points in the plane is not sharp. The methods also say something about sets of $n$ points and $n$ lines that attain the sharp bound of the Szemer\'edi-Trotter point-line incidence bound.
\end{abstract}

\section{Introduction}
The unit distance problem asks for upper bounds on how often the most popular distance can occur in a large finite set of $n$ points in the plane. In 1946, in \cite{Erd46}, Erd\H os conjectured that this could occur no more than $n^{1+\epsilon}$ times for any $\epsilon > 0.$ This was recently disproved by artificial intelligence (see \cite{AI} and \cite{AlonEtAl}), causing an upsurge in interest in the problem. Following this proof, Sawin gave an explicit improvement exhibiting $Cn^{1.014114}$ unit distances for some absolute constant $C>0,$ in\cite{Sawin}. There is still a gap between the best known upper bound of $n^\frac{4}{3}$ due to Spencer, Szemer\'edi, and Trotter in \cite{SST84}, and the recent results.

Because of the recent interest, I have decided to make this document publicly available. This work was mostly done between 2015--2021, but I got stuck. The basic idea is to suppose that $n^\frac{4}{3}$ is tight, then derive a contradiction by following the crossing number lemma backwards. This type of approach worked in the case of distinct dot products, and lead to a concrete improvement for that problem (see \cite{HRS}), so I was hopeful it could work here as well. However, when I saw that in \cite{KS}, Katz and Silier produced results that appear to be significantly stronger that what I proved, I essentially abandoned the attack. Still, here is a description of the approach in case anyone finds it useful.

\subsection{Preliminaries}
By way of contradiction, we suppose that $n^\frac{4}{3}$ is a sharp upper bound for the unit distance problem. We derive necessary facts about any point set attaining this bound. In what follows, we use the notation $X \lesssim Y$ to denote that $X=O(Y),$ $X\approx Y$ to denote that $X=\Theta(Y).$ We also use the fairly standard notation $X= o(Y)$ or $Y=\omega(X).$ Moreover, we use $c$ to denote an undetermined constant, but if it has a subscript, $c_j$, we will keep track of it. Also, we assume that $n$ is a large finite number.

We begin with the celebrated result of Spencer, Szemer\'edi, and Trotter, from \cite{SST84}. Let $u(n)$ denote the maximum number of pairs of points separated by a unit distance in any set of $n$ points in the plane.
\begin{theorem}\label{SST}
$$u(n) \lesssim n^\frac{4}{3}.$$
\end{theorem}
We sketch a of the proof of Theorem \ref{SST} due to Sz\' ekely, from \cite{Szek}. It relies on the following lemma about the crossing number of a graph, which is the smallest number of times edges can cross under any redrawing of the graph in the plane. We we record a statement given in \cite{Szek}, though various statements and proofs exist elsewhere in the literature.
\begin{lemma}\label{crossingNumber}
Given a graph $G(V,E)$ with $|E|>4|V|,$ we have that the crossing number of $G$ is bounded below by
$$cr(G)\geq \frac{|E|^3}{100|V|^2}.$$
\end{lemma}
\begin{proof}
To prove Theorem \ref{SST}, we construct a graph out of any set of $n$ points in the plane. Let the points be the vertices of the graph, so $|V|=n.$ Draw unit circles centered at each point. The arcs of the circles connecting consecutive points will be edges. Notice that by construction, $|E|$ is exactly two times the number of unit distances. So either $|E|\leq 4|V|,$ which implies that the number of unit distances is no more than $4n,$ or we can employ Lemma \ref{crossingNumber}. In the latter case, notice that since the edges are arcs of circles, we can have no more than two crossings from any pair of circles. There are exactly $n(n-1)/2$ pairs of circles, so there can be no more than $n(n-1)$ crossings. We plug in the various quantities and get
\begin{equation}\label{crossingBound}
\frac{|E|^3}{100|V|^2}\leq cr(G)\leq n(n-1),
\end{equation}
which gives us that $|E|\leq (100n^3(n-1))^\frac{1}{3}\lesssim n^\frac{4}{3}.$
\end{proof}

\subsection{Average crossings per edge}
Suppose that there exists a set $P$, of $n$ points in the plane with $c_1n^\frac{4}{3}$ pairs of points separated by a unit distance. We now claim that in the associated graph for such a set, almost every edge must have about $n^\frac{2}{3}$ crossings.
\begin{proposition}\label{avgCrossings}
In the graph associated to the point set, there are no more than $o\left(n^\frac{4}{3}\right)$ edges with $o\left( n^\frac{2}{3}\right)$ crossings or with $\omega\left(n^\frac{2}{3}\right)$ crossings.
\end{proposition}
\begin{proof}
Both claims will follow by way of contradiction. For the first, suppose that there is a subset $E' \subset E$ of $cn^\frac{4}{3}$ edges with $m = o\left(n^\frac{2}{3}\right)$ crossings. Then we could construct a subgraph of $G$ called $G',$ with the same vertex set, but this smaller edge set, $E'.$ We count the total crossings of $G'$ by simply adding up the number of crossings per edge, giving us
$$cr(G') = o\left(|E'| n^\frac{2}{3}\right) =o \left(n^2\right).$$
Notice that we still have $|E'|>4|V|,$ so Lemma \ref{crossingNumber} tells us that
$$cr(G')\gtrsim \frac{|E'|^3}{|V|^2}\approx n^2,$$
which contradicts the upper bound on $cr(G').$
The second claim follows by a similar argument. Suppose there is a subset, $E'' \subset E$ of $cn^\frac{4}{3}$ edges with $m =\omega\left(n^\frac{2}{3}\right)$ crossings. We again count the total number of crossings and get that
$$cr(G)=\omega\left(|E''|n^\frac{2}{3}\right)=\omega\left(n^2\right),$$
but this violates the upper bound of $n(n-1)$ as given in \eqref{crossingBound}.
\end{proof}
\begin{proposition}\label{ppc}
Any set of $n$ points in the plane with $c_1 n^\frac{4}{3}$ unit distances must consist of at least $c_2 n$ points each at a unit distance to at least $c_3 n^\frac{1}{3}$ other points in the set.
\end{proposition}
\begin{proof}
We begin by claiming that there are $\gtrsim n$ circles that each cross $\gtrsim n$ other circles. If this were not the case, we would need to have $o(n)$ circles with $\gtrsim n$ crossings. But then by referring to \eqref{crossingBound}, we could reduce the upper bound on the crossing number, and get that there are fewer than $c_1 n^\frac{4}{3}$ unit distances. Proposition \ref{avgCrossings} has guaranteed that $\gtrsim n^\frac{4}{3}$ of the edges in our graph have $\approx n^\frac{2}{3}$ crossings. So when we count the $\approx n$ total crossings for any of these $\approx n$ circles, they must break down into $\approx n^\frac{1}{3}$ edges with $\approx n^\frac{2}{3}$ crossings each.
\end{proof}
We denote the unit circle centered at the point $p$ by $C(p).$ When two points $q$ and $r$ are less than a unit distance apart, we define $L(q,r)$ to be the union of the two crescent-shaped regions bounded by the $C(q)$ and $C(r).$ We call these regions {\it lunes}. We use the notation $|L(q,r)\cap P|$ to denote the number of points from our set within a given lune. We say that a point $p$ from our set is {\it typical} if there are at least $c_4 n^\frac{1}{3}$ pairs of points $q$ and $r$ that are consecutive on $C(p)$, satisfying
$$c_5 n^\frac{2}{3}\leq |L(q,r)\cap P| \leq c_5'n^\frac{2}{3}.$$
As a consequence of this definition, a point $p$ is typical if and only if there are $\gtrsim n^\frac{1}{3}$ circles passing through $p$ that define lunes with $\approx n^\frac{2}{3}$ points from $P$ properly contained in each of them. We say that a lune {\it emanates} from a point $p$ if the two circles defining the lune intersect at $p$.
\begin{center}
\includegraphics[scale=1]{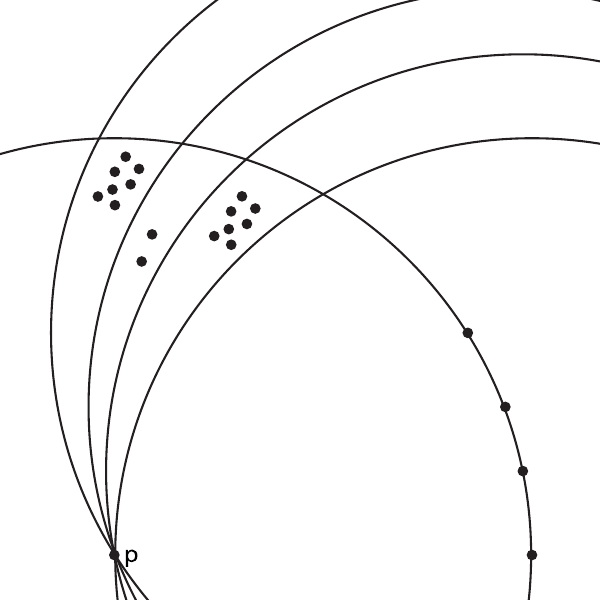}\\
\small{{\bf Figure 1:} There are three lunes pictured, each emanating from $p$. Two of these lunes have many points, while the middle lune has fewer points.}
\end{center}

\begin{proposition}\label{typical}
There are at least $c_6 n$ typical points in $P.$
\end{proposition}
\begin{proof}
By Proposition \ref{ppc}, we get that there are at least $c_2 n$ points $p\in P$ with at least $c_3 n^\frac{1}{3}$ points on $C(p).$ Order these points and call them $q_j.$ Proposition \ref{avgCrossings}, we know that most edges have $\approx n^\frac{2}{3}$ crossings. Consider some $q_j$ and its neighbor $q_{j+1}.$ For a point $x$ to contribute a crossing to the edge connecting $q_j$ and $q_{j+1}$ in the geometric drawing of our graph, an arc of $C(x)$ must cross the arc of $C(p)$ between $q_j$ and $q_{j+1}.$ For this to occur, $x$ must either be within a unit distance of $q_j$ and greater than a unit distance from $q_{j+1},$ or vice versa. This region of possible locations for $x$ is then precisely $L(q_j, q_{j+1}).$ See Figure 1 for an illustration. We note however, that we might not have every single edge on $C(p)$ contribute $\approx n^\frac{2}{3}$ crossings, but by Proposition \ref{avgCrossings} and pigeonholing, there must be $\approx n$ points $p$ with $\approx n^\frac{1}{3}$ edges with $\approx n^\frac{2}{3}$ crossings.
\end{proof}
\begin{center}
\includegraphics[scale=1]{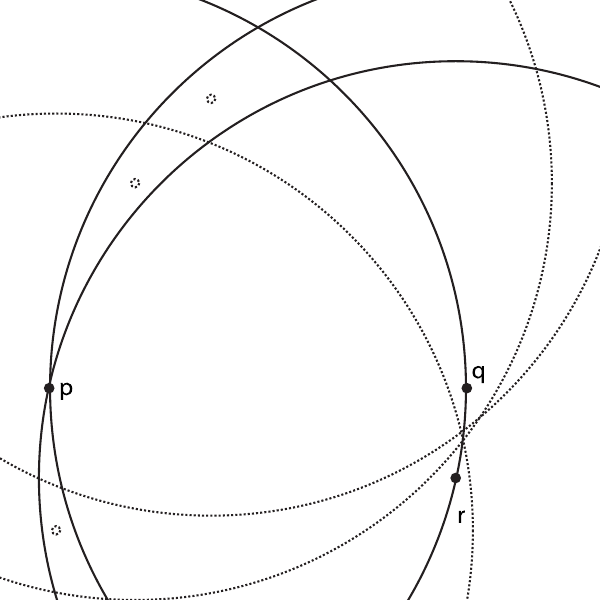}\\
\small{{\bf Figure 2:} The points $q$ and $r$ are at a unit distance from $p$, and hence lie on $C(p).$ The points whose circles cross the edge between $q$ and $r$ are given restricted to the regions between $C(q)$ and $C(r)$, as indicated by the dotted lines.}
\end{center}

\section{Constructing the set}
\subsection{Reducing to a pair of squares}
We begin by following a pigeonholing scheme from Spencer, Szemer\'edi, and Trotter in \cite{SST84}. Suppose that there does exist a set of $n$ points with $\approx n^\frac{4}{3}$ unit distances. Now decompose the plane into disjoint squares with side length $(100\sqrt 2)^{-1}.$ Notice that each circle centered at a point in our set will live on constantly many such squares. Now pick a pair of squares that contributes the most unit distances. Without loss of generality, we can assume that there are $n$ points in each square (if not, we can add up to $n$ points to whichever square has fewer and we still have a set with $cn$ points with enough unit distances), and that the centers of the squares are a unit distance apart. We now restrict our analysis to this pair of squares.

Proposition \ref{typical} tells us that most of our points will live on about $n^\frac{1}{3}$ circles, and these circles will decompose the plane into as many lunes, each with about $n^\frac{2}{3}$ points. Since we are restricting our attention to two squares, we can decompose $P$ into two sets: $A$ will be the set of points of $P$ in one of the squares, while $B$ will be the set of points of $P$ in the other square. We now consider $A\cup B.$ Possibly rotate the point set so that both $A$ and $B$ fit into squares of side-length $1/100$ aligned with the coordinate axes, with the square containing $A$ directly to the left of the square containing $B.$ Call the left square $S_A$ and the right square $S_B.$

\subsection{Introducing $\epsilon$ and $\delta$}
We now look to get a handle on the maximum and minimum distances between points on a given circle. To accomplish this, we will remove points that are too close or too far apart, while retaining a positive proportion of our points and of our unit distances.

Fix positive constants $c_7$ and $c_8$ to be determined later. Define $\delta>0$ to be the minimum height of any horizontal strip $H_\delta$ so that the number of unit distances from points in the set $S_A \cap H_\delta$ to points in $S_B \cap H_\delta$ is at least $c_7 n^\frac{4}{3}.$ Define the point set
$$P_0 := (S_A \cup S_B)\cap H_\delta.$$
Draw unit circles centered at each point in $P_0.$ Proposition \ref{ppc} guarantees that $P_0$ still consists of $\gtrsim n$ points that each lie on $\gtrsim n^\frac{1}{3}$ circles. Arbitrarily label the points of $P_0=\{p_1, p_2, \dots, \}.$ Now, for any $\rho\geq 0$, define $P_\rho$ to be the subset of $P_0$ obtained by pruning points from $P_0$ using the following algorithm. Starting with $p_1,$ consider the points on $C(p_1).$ If any pair of consecutive points on $C(p_1)$ have an arc of length $\leq \rho$ between them, prune the point with lowest index. If $p_2$ was not pruned in the previous step, continue processing points. Notice that any time a point is pruned, we lose $\lesssim n^\frac{1}{3}$ unit distances. Define the constant $\epsilon >0$ to be the largest possible value of $\rho$ we can take such that $P_\rho$ has $\gtrsim c_8n^\frac{4}{3}$ unit distances. Note that $\epsilon$ will depend on our choice of $c_8.$ Finally, we define $E$ to be $P_\epsilon.$

By construction, we see that $E$ is a set of $\gtrsim n$ points determining $\geq c_8 n^\frac{4}{3}$ unit distances that lives within a horizontal strip of height $\delta$, with the property that any unit circle centered at a point in $E$ will have circular arcs between points of length at least $\epsilon.$ Again, by Proposition \ref{ppc}, we have that there must be $\gtrsim n^\frac{1}{3}$ points on most of the circles centered at points in $E$, and consecutive pairs along these circular arcs must be separated by arcs of length at least $\epsilon,$ so we can see that $\delta \gtrsim \epsilon n^\frac{1}{3}.$

\begin{center}
\includegraphics[scale=1]{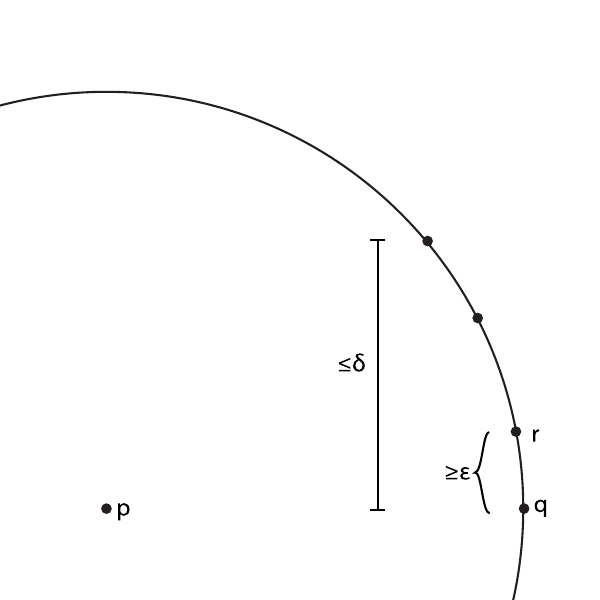}\\
\small{{\bf Figure 3:} The points $q$ and $r$ are at a unit distance from $p$, and hence lie on $C(p).$ The length of the arc between $q$ and $r$ is at least $\epsilon,$ and the points all fit within a horizontal strip of height $\delta.$}
\end{center}

\subsection{Summarizing}
Any of the $\gtrsim n$ typical points $p\in E\cap A$ will essentially partition the other points of $A$ into sets of $\approx n^\frac{2}{3}$ points, with each set in the interior of a different lune emanating from $p$. Moreover, each of the $\approx n^\frac{1}{3}$ circles through $p$ will have $\approx n^\frac{1}{3}$ points from $E.$ Moreover, these circles will meet at $p$ in angles $\geq \epsilon,$ because two circles through $p$ are centered at points on $C(p)$ which must have an arc of length at least $\epsilon$ between them.

\section{An incomplete attempt at a contradiction}
What follows is an incomplete attempt to show that the constraints above will lead to a contradiction. The basic idea is that any set of $n$ points that determines $\gtrsim n^\frac{4}{3}$ unit distances will necessarily have two tightly clustered subsets of $\gtrsim n$ points each, with the properties detailed above. Next, we find the two points, say $x$ and $y,$ that are closest along any relevant circular arc. Call the arc length between them $\epsilon,$ and note that their circles, $C(x)$ and $C(y),$ must meet at the smallest angle (also $\epsilon$) of any pair of relevant crossing circles. We notice that there still must be $\gtrsim n^\frac{2}{3}$ points in the lune defined between $C(x)$ and $C(y),$ and hope that this clustering of points leads to another pair of points, say $x'$ and $y'$, that are then even closer than $\epsilon$ along some relevant circular arc. In what follows, we record some potentially helpful trigonometric estimates.

\subsection{Basic facts from trigonometry}
Recall the Taylor expansion for sine and cosine,
$$\sin x = x-\frac{x^3}{3!}+\frac{x^5}{5!}-\dots$$
$$\cos x = 1-\frac{x^2}{2!}+\frac{x^4}{4!}-\frac{x^6}{6!}+\dots$$

\begin{proposition}\label{distArc}
Given distinct points $a$ and $b$ on a unit circle separated by an arc of length $\theta <\frac{\pi}{2}:$
\begin{enumerate}
\item[(i)] The distance from $a$ to $b$ is $\sqrt{2-2\cos\theta}.$
\item[(ii)] $\theta - \frac{\theta^2}{\sqrt{12}} < \sqrt{2-2\cos\theta} < \theta - \frac{\theta^3}{25}$
\end{enumerate}
\end{proposition}
\begin{proof}
$(i)$ The arc of the circle subtends the angle $\theta,$ so we can directly apply the law of cosines to the triangle formed by $a, b,$ and the center of the circle.

$(ii)$
With this in mind, estimate from below by
$$
\sqrt{2-2\cos \theta} \geq \sqrt{2-2\left(1-\frac{\theta^2}{2!}+\frac{\theta^4}{4!}\right)}=\sqrt{\theta^2-\frac{\theta^4}{12}}
$$
$$
=\sqrt{\left(\theta-\frac{\theta^2}{\sqrt{12}}\right)\left(\theta+\frac{\theta^2}{\sqrt{12}}\right)}> \sqrt{\left(\theta-\frac{\theta^2}{\sqrt{12}}\right)^2}.
$$
Similarly, we estimate from above by
$$
\sqrt{2-2\cos \theta} \leq \sqrt{2-2\left(1-\frac{\theta^2}{2!}+\frac{\theta^4}{4!}-\frac{\theta^6}{6!}\right)}
=\sqrt{\theta^2-\frac{\theta^4}{12}+\frac{\theta^6}{360}}
$$
$$
< \sqrt{\theta^2-\frac{2\cdot\theta^4}{25}+\frac{\theta^6}{625}} = \sqrt{\left(\theta - \frac{\theta^3}{25}\right)^2},
$$
where in the last line, we used the fact that $\theta<\pi/2< \sqrt{2250/795}.$
\end{proof}

\begin{center}
\includegraphics[scale=1]{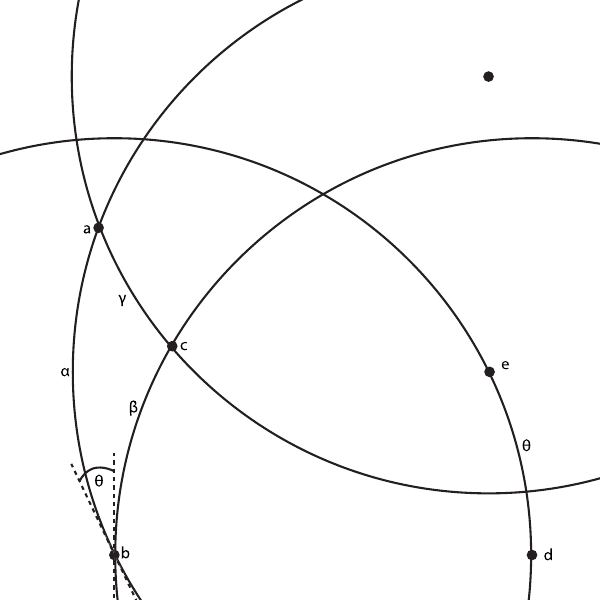}\\
\small{{\bf Figure 4:} Setting for Proposition \ref{curvTri}.}
\end{center}

We next give an estimate for certain arcs of circles that will come up in the calculations to follow. To get a handle on the arcs in question, we consider certain types of curvilinear triangles. If we are know two arc lengths, we can estimate the third.
\begin{lemma}\label{curvTri}
Suppose that each pair of the points $a, b,$ and $c$ is on a different unit circle, with the length of the arc from $a$ to $b$ measuring $\alpha<\pi/100,$ and the length of the arc from $b$ to $c$ measuring $\beta<\pi/100,$ and the lines tangent to the circles meeting at $b$ have an angle of $\theta<\pi/100.$ Then the distance from $a$ to $c$ is
$$|a-c| = \left(\frac{1}{2}\left|\alpha^2-2\alpha\theta-\beta^2\right|+\left|\alpha-\beta\right|\right)+R,$$
where $|R|\lesssim (\alpha+\beta+\theta)^3.$
\end{lemma}
The main takeaway of this result for the task at hand is that there are essentially two dominant behaviors in the relevant range of parameters. If $\alpha$ and $\beta$, are nearly equal (that is, if $|\alpha-\beta|\lesssim |\alpha\theta|$), then $|a-c|$ behaves like an arc of a circle of radius $\alpha$ subtending an angle $\theta.$ If $\alpha$ and $\beta$ are sufficiently different (that is, if $|\alpha-\beta|\gtrsim |\alpha\theta|$), then $|a-c|$ looks more like their difference.
\begin{proof}
Let $d$ and $e$ be points on $C(b),$ so that $c\in C(d)$ and $a\in C(e).$ Without loss of generality, we can assume that $b$ is the origin, $d=(1,0),$ and $e=(\cos\theta, \sin \theta).$ From here, we get that $c=(1-\cos\beta,\sin \beta),$ and $a=(\cos\theta-\cos(\alpha-\theta),\sin\theta+\sin(\alpha-\theta))$. We then we write the distance squared in terms of the horizontal difference, which we denote $X$ and the vertical difference, which we call $Y$.
\begin{align*}
|a-c|^2&=|\cos\theta-\cos(\alpha-\theta)-1-\cos\beta|^2+|\sin\theta+\sin(\alpha-\theta)-\sin \beta|^2\\
&= X^2+Y^2.
\end{align*}
We now write $X$ using the first three terms of the Taylor expansion for cosine and a small error, $R_X$, where $|R_X|\leq c_{10}(\alpha+\beta+\theta)^6,$ for some $c_{10}\in [0,1/100].$
\begin{align*}
X&=|\cos\theta-\cos(\alpha-\theta)-(1-\cos\beta)|\\
&= \left|\left(1-\frac{\theta^2}{2} +\frac{\theta^4}{24}\right)- \left(1-\frac{(\alpha-\theta)^2}{2} +\frac{(\alpha-\theta)^4}{24}\right)-1+\left(1-\frac{\beta^2}{2}+\frac{\beta^4}{24} \right)\right|+R_X\\
&=\left|-\frac{\theta^2}{2} +\frac{\theta^4}{24}+\frac{(\alpha-\theta)^2}{2} -\frac{(\alpha-\theta)^4}{24}-\frac{\beta^2}{2}+\frac{\beta^4}{24}\right|+R_X\\
&=\left|\frac{\alpha^2-2\alpha\theta}{2}-\frac{\beta^2}{2}+\frac{\theta^4}{24} -\frac{(\alpha-\theta)^4}{24}+\frac{\beta^4}{24}\right|+R_X\\
&= \left|\frac{\alpha^2-2\alpha\theta-\beta^2}{2}\right|+R_X',\\
\end{align*}
where $|R_X'|\leq c_{11}(\alpha+\beta+\theta)^4,$ for some $c_{11}\in [0,1/10].$ Write $Y$ using the first three terms of the Taylor expansion for sine and another small error, $R_Y$, where $|R_Y|\leq c_{12}(\alpha+\beta+\theta)^5,$ for some $c_{12}\in [0,1/100].$
\begin{align*}
Y&=|\sin\theta+\sin(\alpha-\theta)-\sin \beta|\\
&=\left| \left(\theta-\frac{\theta^3}{6} \right)+\left((\alpha-\theta)-\frac{(\alpha-\theta)^3}{6} \right)-\left(\beta-\frac{\beta^3}{6} \right)\right|+R_Y\\
&=\left| \theta-\frac{\theta^3}{6} +(\alpha-\theta)-\frac{(\alpha-\theta)^3}{6} -\beta+\frac{\beta^3}{6}\right|+R_Y\\
&=\left|\alpha-\beta-\frac{\theta^3}{6}-\frac{(\alpha-\theta)^3}{6} +\frac{\beta^3}{6}\right|+R_Y\\
&=\left|\alpha-\beta\right|+R_Y',\\
\end{align*}
where $|R_Y'|\leq c_{13}(\alpha+\beta+\theta)^3,$ for some $c_{13}\in [0,1/3].$
Plugging in the expressions for $X$ and $Y$ we get
\begin{align*}
|a-c|&=\sqrt{X^2+Y^2} = \sqrt{\left(\left|\frac{\alpha^2-2\alpha\theta-\beta^2}{2}\right|+R_X'\right)^2+\left(\left|\alpha-\beta\right| +R_Y'\right)^2}\\
&=\frac{1}{2}\left|\alpha^2-2\alpha\theta-\beta^2\right|+\left|\alpha-\beta\right|+R,
\end{align*}
where $|R|\lesssim (\alpha + \beta + \theta)^3.$
\end{proof}

One clear consequence of Lemma \ref{curvTri} is the following.
\begin{corollary}\label{thirdArc}
Suppose that each pair of the points $a, b,$ and $c$ is on a different unit circle, with the length of the arc from $a$ to $b$ measuring $\alpha<10^{-10},$ and the length of the arc from $b$ to $c$ measuring $\beta<10^{-10},$ and the lines tangent to the circles meeting at $b$ have an angle of $\theta<10^{-10}.$ Then the distance from $a$ to $c$ is
$$
|a-c| + R'=\left\{\begin{tabular}{l l}$|\alpha-\beta|,$& $|\alpha-\beta|\geq \alpha\theta$\\ $|\alpha\theta|,$& $|\alpha-\beta|\leq \alpha\theta$\\ \end{tabular}\right.,
$$
where $|R'|\lesssim (\alpha + \beta)^2.$
\end{corollary}
\begin{proof}
By applying Lemma \ref{curvTri}, we get that
$$|a-c| = \left(\frac{1}{2}\left|\alpha^2-2\alpha\theta-\beta^2\right|+\left|\alpha-\beta\right|\right)+R,$$
where $|R|\lesssim (\alpha+\beta+\theta)^3.$ Now compare the two dominant terms:
$$\frac{1}{2}\left|\alpha^2-2\alpha\theta-\beta^2\right| = \left|\alpha-\beta\right|.$$
Notice that equality holds when
$$\alpha\theta = |\alpha-\beta|+\frac{|\alpha^2-\beta^2|}{2},$$
and the desired result follows
\end{proof}

\section{Szemer\'edi-Trotter sharpness}
We begin with the celebrated result of Szemer\'edi, and Trotter, from \cite{ST83}. Given a point $p$ and a line $\ell,$ the pair $(p,\ell)$ is called an {\bf incidence} if $p\in\ell.$ Given a set of points $P$ and a set of lines $L$, we often care about the total number of incidences between the points from $P$ and the lines from $L$, denoted $I(P,L).$
\begin{theorem}\label{ST}
Given any set of points $P$ and set of lines $L$ in $\mathbb R^2,$ we have
$$I(P,L) \lesssim (|P||L|)^\frac{2}{3}+|P|+|L|.$$
\end{theorem}
We sketch a of the proof of Theorem \ref{ST} due to Sz\' ekely, from \cite{Szek}. It relies on Lemma \ref{crossingBound} above.
\begin{proof}
To prove Theorem \ref{ST}, we construct a graph out of any set of $n$ points $P$ and any set of $m$ lines $L$. We will assume that every point from $P$ is on at least one line from $L$, and every line from $L$ has at least one from $P$, as the bound is even sharper otherwise. Now let the points be the vertices of the graph, so $|V|=n.$ The line segments connecting consecutive points will be edges. Because every line with $k$ points on it contributes $k-1$ edges, the edges are exactly the total number of incidences less the total number of lines. Therefore, $|E|$ is exactly $I-|L|$. Now either $|E|\leq 4|V|,$ which implies that the number of incidences is no more than $\lesssim|P|,$ or we can employ Lemma \ref{crossingNumber}. In the latter case, notice that since the edges are line segments that come from the lines in $L$, so we can have no more than one crossing from any pair of lines. There are exactly $\leq|L|^2$ pairs of circles, so there can be no more than $|L|^2$ crossings. We plug in the various quantities and get
\begin{equation}\label{crossingBound}
\frac{|E|^3}{100|V|^2}\leq cr(G)\leq |L|^2,
\end{equation}
which gives us that $|E|=(I-|L|)\lesssim (|L||P|)^\frac{2}{3},$
yielding the desired bound.
\end{proof}
We call a pair of sets $(P,L)$ where $P$ is a set of points and $L$ is a set of lines a {\bf ST-sharp} pair if they have $I(P,L)\gtrsim (|P||L|)^\frac{2}{3}.$

\subsection{Average crossings per edge and consequences}
Consider a pair of sets in $\mathbb R^2$: $P$, consisting of $n$ points and $L,$ consisting of $n$ lines in the plane with $\gtrsim n^\frac{4}{3}$ incidences. We now claim that in the associated graph for such a set, almost every edge must have about $n^\frac{2}{3}$ crossings.
\begin{proposition}\label{avgCrossings}
In the graph associated to the point set, there are no more than $o\left(n^\frac{4}{3}\right)$ edges with $o\left( n^\frac{2}{3}\right)$ crossings or with $\omega\left(n^\frac{2}{3}\right)$ crossings.
\end{proposition}
\begin{proof}
Both claims will follow by way of contradiction. For the first, suppose that there is a subset $E' \subset E$ of $cn^\frac{4}{3}$ edges with $m = o\left(n^\frac{2}{3}\right)$ crossings. Then we could construct a subgraph of $G$ called $G',$ with the same vertex set, but this smaller edge set, $E'.$ We count the total crossings of $G'$ by simply adding up the number of crossings per edge, giving us
$$cr(G') = o\left(|E'| n^\frac{2}{3}\right) =o \left(n^2\right).$$
Notice that we still have $|E'|>4|V|,$ so Lemma \ref{crossingNumber} tells us that
$$cr(G')\gtrsim \frac{|E'|^3}{|V|^2}\approx n^2,$$
which contradicts the upper bound on $cr(G').$
The second claim follows by a similar argument. Suppose there is a subset, $E'' \subset E$ of $cn^\frac{4}{3}$ edges with $m =\omega\left(n^\frac{2}{3}\right)$ crossings. We again count the total number of crossings and get that
$$cr(G)=\omega\left(|E''|n^\frac{2}{3}\right)=\omega\left(n^2\right),$$
but this violates the upper bound of $n(n-1)$ as given in \eqref{crossingBound}.
\end{proof}
\begin{proposition}\label{ppl}
Any pair of sets of $n$ points and $n$ lines in the plane with $\gtrsim n^\frac{4}{3}$ incidences must consist of at least $\gtrsim n$ points lying on $\gtrsim n^\frac{1}{3}$ lines, and $\gtrsim n$ lines with $\gtrsim n^\frac{1}{3}$ points each.
\end{proposition}
\begin{proof}
We begin by claiming that there are $\gtrsim n$ lines that each cross $\gtrsim n$ other lines. If this were not the case, we would need to have $o(n)$ lines with $\gtrsim n$ crossings. But then by referring to \eqref{crossingBound}, we could reduce the upper bound on the crossing number, and get that there are fewer than $o\left(n^\frac{4}{3}\right)$ incidences. Proposition \ref{avgCrossings} has guaranteed that $\gtrsim n^\frac{4}{3}$ of the edges in our graph have $\approx n^\frac{2}{3}$ crossings. So when we count the $\approx n$ total crossings for any of these $\approx n$ lines, they must break down into $\approx n^\frac{1}{3}$ edges with $\approx n^\frac{2}{3}$ crossings each.
\end{proof}

The fact that the crossings per edge behave in this way gives rise to another fact about the distribution of points and lines.

\begin{proposition}\label{propellerProp}
Any pair of sets of $n$ points and $n$ lines in the plane with $\gtrsim n^\frac{4}{3}$ incidences must consist of at least $\gtrsim n$ points $P'\subseteq P,$ lying on $\gtrsim n^\frac{1}{3}$ lines, with $\approx n^\frac{2}{3}$ points lying between $\gtrsim n^\frac{1}{3}$ pairs of consecutive lines through at least $\gtrsim n$ of the points $p\in P'.$ 
\end{proposition}
\begin{proof}
To prove this result, we introduce some notation. Given a point $p\in \mathbb R^2\setminus \{(0,0)\},$ define $\ell(p)$ to be the set of points that have dot product $1$ with the point $p.$ It is elementary to show that this will be a line perpendicular to the direction from the origin to $p.$ Given a set of points $P$, consider its dual, written
\[\hat P := \{\ell(p):p\in P\}.\]
Similarly, for a line $\ell$ not through the origin, we denote the unique point that has dot product 1 with all points on $\ell$ by $p(\ell)$. Given a set of lines $L$, we also consider its dual, denoted
\[\hat L := \{p(\ell):\ell\in L\}.\]
Given $p\in P,$ let $L(p),$ the set of lines in $L$ passing through $p.$ Similarly, given a line $\ell\in L,$ let $P(\ell)$ denote the points in $P$ lying on $\ell.$
Now, given an ST-sharp pair, $(P,L)$ with both sets of size $n$, we consider the dual pair, $(\hat L, \hat P).$ By Proposition \ref{ppl}, we have a set $P'$ of $\gtrsim n$ points, each lying on $\gtrsim n^\frac{1}{3}$ lines. Next we define $L'$ to be the set of lines through points in $P'$ with at least $n^\frac{1}{3}$ points from $P,$ namely\footnote{We might need to pay attention to constants here so we aren't overpruning. What I mean is, we are claiming here that there are many points on many lines with many other points, but if we want to modify the definition of $L'$ to have $\ell \cap P',$ there might be some issue with constants collapsing too quickly as we specify subsets.}
\[L':=\{\ell\in L(p): p\in P', |\ell \cap P|\gtrsim n^\frac{1}{3}\}.\]
Fix one of the $\gtrsim n$ points $p\in P'.$ It has $\gtrsim n^\frac{1}{3}$ lines from $L$ running through it, so $|L(p)|\gtrsim n^\frac{1}{3}$. Now consider the set of duals of these lines, $\widehat{L(p)}.$ Each element $q\in \widehat{L(p)}$ is a point that determines a dot product 1 with a line from $L(p).$ Because each of the lines in $L(p)$ goes through $p$, we can see that each of these points $q\in \widehat{L(p)}$ will determine a dot product with $p$, they all lie on $\ell(p)$ by definition.
\begin{center}\label{fig5}
\includegraphics[scale=1]{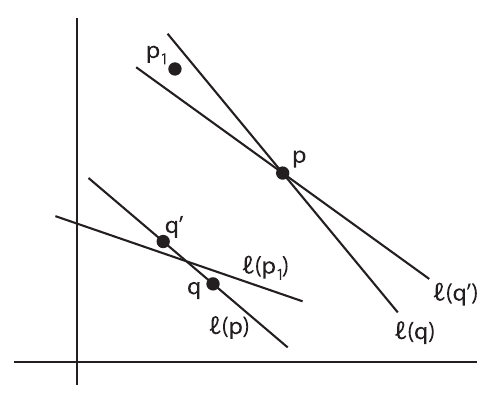}\\
\small{{\bf Figure 5:} The point $p\in P$ is on two lines from $L$. Each of these lines gives rise to a point in $\hat L$, represented by $q$ and $q'$, respectively. Moreover, those points are on the line $\ell(p)\in\hat P.$ Finally, the point $p_1$ is between the lines $\ell(q)$ and $\ell(q')$ because $\ell(p_1)$ crosses $\ell(p)$ between $q$ and $q'.$}
\end{center}
Notice that if $p\in \ell,$ then $p(\ell)\in \ell(p).$ So we have that $I(P,L) = I(\hat L, \hat P),$ and therefore $(\hat L, \hat P)$ is also an ST-sharp pair. So by the reasoning in the proof of Proposition \ref{avgCrossings}, we corresponding graph for $(\hat L, \hat P)$ must have $\approx n$ edges with $\approx n^\frac{2}{3}$ crossings. Therefore there must be $\gtrsim n^\frac{4}{3}$ pairs of points $q,q'\in \hat L$ so that $q$ and $q'$ that are consecutive on $\ell(p)$ for some $p\in P'$ and have $\gtrsim n^\frac{2}{3}$ lines crossing the segment between them. These lines must have the form $\ell(p_j)$ for $\gtrsim n^\frac{2}{3}$ choices of points $p_j\in P.$
To conclude, we fix our attention upon one such pair of $(q,q')$ and one such $p_1.$ Because $\ell(p_1)$ does not contain $q$ or $q',$ but crosses $\ell(p)$ between them, we know that (without loss of generality) $p_1\cdot q <1$ and $p_1\cdot q'>1.$ Therefore, $p_1$ must be strictly between the lines $\ell(q)$ and $\ell(q').$ See Figure 5. If the opposite inequalities hold, then $p_1$ is still between the lines $\ell(q)$ and $\ell(q'),$ albeit on the other side of $p$. See Figure 6.
\begin{center}
\includegraphics[scale=1]{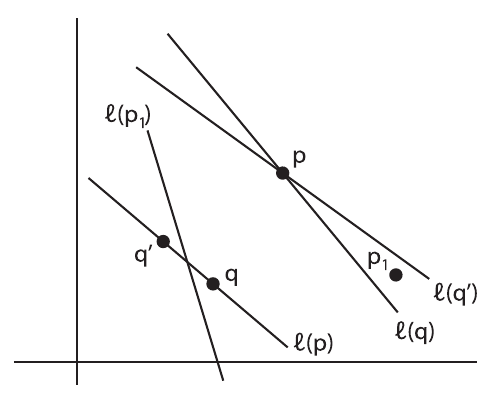}\\
\small{{\bf Figure 6:} This is the same scenario as before, except with $p_1\cdot q' <1$ and $p_1\cdot q>1.$}
\end{center}
By running over all $\gtrsim n^\frac{2}{3}$ points $p_j$ for this pair $(q,q'),$ we see that each $p_j$ must be between the lines $\ell(q)$ and $\ell(q')$. Finally, observing that this holds for $\gtrsim n^\frac{1}{3}$ consecutive pairs $(q,q')$ on $\ell(p)$ for all points $p\in P'$ yields the desired result.
\end{proof}

\end{document}